\documentclass[twoside,leqno,10pt, A4]{amsart}
\usepackage{amsfonts}
\usepackage{amsmath}
\usepackage{amscd}
\usepackage{amssymb}
\usepackage{amsthm}
\usepackage{amsrefs}
\usepackage{latexsym}
\usepackage{mathrsfs}
\usepackage{bbm}
\usepackage{enumerate}
\usepackage{graphicx}

\usepackage{amsfonts}
\usepackage{amsmath}
\usepackage{amscd}
\usepackage{amssymb}
\usepackage{amsthm}
\usepackage{amsrefs}
\usepackage{latexsym}
\usepackage{mathrsfs}
\usepackage{bbm}
\usepackage{amscd}
\usepackage{amssymb}
\usepackage{amsthm}
\usepackage{amsrefs}
\usepackage{latexsym}
\usepackage{mathrsfs}
\usepackage{bbm}
\usepackage{enumerate}
\usepackage{graphicx}
\usepackage{color}
\begin{document}

\newtheorem{theorem}[subsection]{Theorem}
\newtheorem{proposition}[subsection]{Proposition}
\newtheorem{lemma}[subsection]{Lemma}
\newtheorem{corollary}[subsection]{Corollary}
\newtheorem{conjecture}[subsection]{Conjecture}
\newtheorem{prop}[subsection]{Proposition}
\numberwithin{equation}{section}
\newcommand{\mr}{\ensuremath{\mathbb R}}
\newcommand{\mc}{\ensuremath{\mathbb C}}
\newcommand{\dif}{\mathrm{d}}
\newcommand{\intz}{\mathbb{Z}}
\newcommand{\ratq}{\mathbb{Q}}
\newcommand{\natn}{\mathbb{N}}
\newcommand{\comc}{\mathbb{C}}
\newcommand{\rear}{\mathbb{R}}
\newcommand{\prip}{\mathbb{P}}
\newcommand{\uph}{\mathbb{H}}
\newcommand{\fief}{\mathbb{F}}
\newcommand{\majorarc}{\mathfrak{M}}
\newcommand{\minorarc}{\mathfrak{m}}
\newcommand{\sings}{\mathfrak{S}}
\newcommand{\fA}{\ensuremath{\mathfrak A}}
\newcommand{\mn}{\ensuremath{\mathbb N}}
\newcommand{\mq}{\ensuremath{\mathbb Q}}
\newcommand{\half}{\tfrac{1}{2}}
\newcommand{\f}{f\times \chi}
\newcommand{\summ}{\mathop{{\sum}^{\star}}}
\newcommand{\chiq}{\chi \bmod q}
\newcommand{\chidb}{\chi \bmod db}
\newcommand{\chid}{\chi \bmod d}
\newcommand{\sym}{\text{sym}^2}
\newcommand{\hhalf}{\tfrac{1}{2}}
\newcommand{\sumstar}{\sideset{}{^*}\sum}
\newcommand{\sumprime}{\sideset{}{'}\sum}
\newcommand{\sumprimeprime}{\sideset{}{''}\sum}
\newcommand{\sumflat}{\sideset{}{^\flat}\sum}
\newcommand{\shortmod}{\ensuremath{\negthickspace \negthickspace \negthickspace \pmod}}
\newcommand{\V}{V\left(\frac{nm}{q^2}\right)}
\newcommand{\sumi}{\mathop{{\sum}^{\dagger}}}
\newcommand{\mz}{\ensuremath{\mathbb Z}}
\newcommand{\leg}[2]{\left(\frac{#1}{#2}\right)}
\newcommand{\muK}{\mu_{\omega}}
\newcommand{\thalf}{\tfrac12}
\newcommand{\lp}{\left(}
\newcommand{\rp}{\right)}
\newcommand{\Lam}{\Lambda_{[i]}}
\newcommand{\lam}{\lambda}
\def\L{\fracwithdelims}
\def\om{\omega}
\def\pbar{\overline{\psi}}
\def\phi{\varphi}
\def\lam{\lambda}
\def\lbar{\overline{\lambda}}
\newcommand\Sum{\Cal S}
\def\Lam{\Lambda}
\newcommand{\sumtt}{\underset{(d,2)=1}{{\sum}^*}}
\newcommand{\sumt}{\underset{(d,2)=1}{\sum \nolimits^{*}} \widetilde w\left( \frac dX \right) }

\theoremstyle{plain}
\newtheorem{conj}{Conjecture}
\newtheorem{remark}[subsection]{Remark}

\makeatletter
\def\widebreve{\mathpalette\wide@breve}
\def\wide@breve#1#2{\sbox\z@{$#1#2$}%
     \mathop{\vbox{\m@th\ialign{##\crcr
\kern0.08em\brevefill#1{0.8\wd\z@}\crcr\noalign{\nointerlineskip}%
                    $\hss#1#2\hss$\crcr}}}\limits}
\def\brevefill#1#2{$\m@th\sbox\tw@{$#1($}%
  \hss\resizebox{#2}{\wd\tw@}{\rotatebox[origin=c]{90}{\upshape(}}\hss$}
\makeatletter

\title[Sharp upper bounds for moments of quadratic Dirichlet $L$-functions]{Sharp upper bounds for moments of quadratic Dirichlet $L$-functions}

%%\date{\today}
\author{Peng Gao}
\address{School of Mathematical Sciences, Beihang University, Beijing 100191, P. R. China}
\email{penggao@buaa.edu.cn}
\begin{abstract}
 We establish unconditional sharp upper bounds of the $k$-th moments of the family of quadratic Dirichlet $L$-functions at the central point for $0 \leq k \leq 2$.
\end{abstract}

\maketitle



\section{Introduction}
\label{sec 1}

  Moments of central values of families of $L$-functions have attracted much attention in research as they can be applied to address the non-vanishing issue of these values, which in turn carry significant arithmetic implications. Much progress has been made in late years that largely enhances our understanding of these moments.  A conjecture concerning asymptotic expressions for the moments of various families of $L$-functions has been made by J. P. Keating and N. C. Snaith in \cite{Keating-Snaith02}, in connection with random matrix theory. Another conjecture of J. B. Conrey, D. W. Farmer, J. P. Keating, M. O. Rubinstein and N. C. Snaith in \cite{CFKRS} gives more precise predictions by including lower order terms on the asymptotic behaviors of the moments for certain families of $L$-functions.

   In \cite{R&Sound} and \cite{R&Sound1}, Z. Rudnick and K. Soundararajan developed a simple and powerful method towards establishing lower bounds for rational moments of families of $L$-functions of the conjectured order of magnitude. The method is further extended by M. Radziwi{\l\l} and K. Soundararajan in \cite{Radziwill&Sound1} to all moments larger than the first. On the other hand, an approach due to K. Soundararajan in \cite{Sound01} enables one to obtain the corresponding upper bounds under the generalized Riemann hypothesis (GRH). This approach was further sharpened by A. J. Harper \cite{Harper} to give upper bounds of desired order of magnitude for moments of $L$-functions in many families conditionally.

  In \cite{Radziwill&Sound}, M. Radziwi{\l\l} and K. Soundararajan developed a new principle that allows one to seek upper bounds of all smaller
moments with the knowledge of an upper bound for a particular moment. This principle is further implemented by W. Heap and K. Soundararajan in \cite{H&Sound} to treat the case of lower bounds.

  In this paper, we are interested in the family of quadratic Dirichlet $L$-functions. Asymptotic formulas for the first two moments of this family was obtained by M. Jutila in \cite{Jutila} with the error terms being subsequently improved in \cites{ViTa, DoHo, Young1, sound1, Sono}.  For the purpose of this paper, we are interested in the family given by $\{ L(s, \chi_{8d}) \}$ for $d$ being odd and square-free integers.  Here $\chi_{8d} = \left(\frac{8d}{\cdot} \right)$ is the Kronecker symbol. The study of the above family was initiated by K. Soundararajan in \cite{sound1}, who obtained mollified first two moments of the family to show that at least $87.5\%$ of the members of this family have non-vanishing central values. The third moment of the above family is also obtained for the first time in \cite{sound1}. The error term was improved by M. P. Young in \cite{Young2} for a smoothed version. In \cite{Shen}, Q. Shen obtained the fourth moment of the above family under GRH.

  Asymptotic formulas for all positive real moments of $\{ L(\half, \chi_{8d}) \}$ are conjectured by J. C. Andrade and Keating, J. P. in \cite{Andrade-Keating01}.  Combining the above mentioned work of \cite{Harper, Sound01, R&Sound, R&Sound1, Radziwill&Sound1}, we now know that
\begin{align*}
%%\label{upperlower}
 X(\log X)^{\frac{k(k+1)}{2}} \ll_k \sumstar_{\substack{ 0<d<X \\ (d,2)=1}}|L(\tfrac{1}{2},\chi_{8d})|^k \ll_k X(\log X)^{\frac{k(k+1)}{2}},
\end{align*}
 where the lower bound above holds for all real $k \geq 1$ unconditionally and the upper bound above holds for all $k \geq 0$ under GRH. Here and throughout the paper, we denote $\sumstar$ for the sum over square-free integers.

  It is our goal in this paper to apply the principal of obtaining upper bounds given in \cite{Radziwill&Sound} by M. Radziwi{\l\l} and K. Soundararajan in our setting. Before we state our result, we would like to recall the principle. Although the principle works for general $L$-functions, we only adapt it for the family $\{ L(s, \chi_{8d}) \}$ for simplicity. For this, we recall that a result of B. Hough \cite{Hough} shows that under GRH,  the quantity $\log |L(\half, \chi_{8d})|$ with $X < d \leq 2X$ for large $X$ is normally distributed with mean $\half \log \log d$ and variance $\log \log d$. On the other hand, it is shown in  \cite{Hough} that the sum
\begin{align}
\label{lambdasum}
 \sum_{n \leq z}\frac{\Lambda(n)\chi_{8d}(n)}{n^{\frac{1}{2}}}
\end{align}
is similarly distributed as $\log |L(\half, \chi_{8d})|$ for $z=X^{1/(\log \log X)^2}$, say. Here $\Lambda(n)$ is the von Mangoldt function. As the contribution of prime powers $\geq 3$ is negligible in the above sum and the contribution of prime squares is about $\log \log X$, we see that
the difference between the expression given in \eqref{lambdasum} and $\log \log X$ is mainly determined by
\begin{align*}
%%\label{defPd}
{\mathcal P}(d) = \sum_{p \leq z} \frac{1}{\sqrt{p}} \chi_{8d}(p).
\end{align*}
 Taking exponentials, this implies that the quantity $|L(\half, \chi_{8d})|(\log |d|)^{1/2}\exp(-{\mathcal P}(d))$ is usually small. Thus, for any given real numbers $n>0, 0 < k < 1$, we may estimate the quantity $(|L(\half, \chi_{8d})|(\log |d|^{1/2})^{nk}$ by writing it as
\begin{align*}
%%\label{defPd}
\Big (|L(\half, \chi_{8d})|(\log |d| )^{1/2}\Big )^{nk}=|L(\half, \chi_{8d})|^{nk}(\log |d|)^{nk/2}\exp(-nk(1-k){\mathcal P}(d))\cdot \exp(nk(1-k){\mathcal P}(d)).
\end{align*}
 We now recall that Young's inequality asserts that we have $ab\le a^p/p + b^q/q$ for real numbers $a, b \geq 0$ and real numbers $p, q \ge 1$ satisfying $1/p+1/q=1$.  Applying this with $a=|L(\half, \chi_{8d})|^{nk}(\log |d|)^{nk/2}\exp(-nk(1-k){\mathcal P}(d)), b=\exp(nk(1-k){\mathcal P}(d))$ and $p=1/k, q=1/(1-k)$, we see that
\begin{align}
\label{Basicapproach}
\Big (|L(\half, \chi_{8d})|(\log |d| )^{1/2}\Big )^{nk} \leq k|L(\half, \chi_{8d})|^{n}(\log |d|)^{n/2}\exp(-n(1-k){\mathcal P}(d))+(1-k)\exp(kn{\mathcal P}(d)).
\end{align}
  As we expect  $|L(\half, \chi_{8d})|^n(\log |d|)^{n/2}\exp(-n{\mathcal P}(d))$ to be small most of the time, the right side of \eqref{Basicapproach} should be bounded above by $\exp(kn{\mathcal P}(d))$ on average. As expanding $\exp(kn{\mathcal P}(d))$ into Euler products leads to a too long Dirichlet polynomial to estimate, we approximate it by taking a suitably long Taylor expansion to achieve our goal since ${\mathcal P}(d)$ is often small in size.

  We note that the bound given in \eqref{Basicapproach} may be modified to establish similar bounds for general $L$-functions. If we further know that the corresponding $L$-function satisfies $L(\half) \geq 0$, then we may replace $|L(\half)|^n$ on the right side of \eqref{Basicapproach} by $L(\half)^n$ to see that if we have a good understanding of the $n$-th moment (twisted by another character) of the corresponding family of $L$-functions at the central value, then we shall be able to obtain sharper upper bounds for every $m$-th moment of the same family of $L$-functions with $0 \leq m \leq n$. This is precisely what has been carried out by Radziwi{\l\l} and Soundararajan in \cite{Radziwill&Sound} to treat the moments of quadratic twists of $L$-functions attached to elliptic curves, since in this case it is known that the corresponding $L$-functions have non-negative values at $1/2$ and the first moment of the family can be evaluated.

  For the family of quadratic twists of Dirichlet $L$-functions, although our current stage of knowledge is short of determining whether $L(\half, \chi_{8d}) \geq 0$ is always valid, we do know that these values are real so that $L(\half, \chi_{8d})^2 \geq 0$ is always true. Thanks to the work of K. Soundararajan in \cite{sound1}, we also have a good knowledge of the twisted second moment of the same family.  Combining these with the above mentioned principle of Radziwi{\l\l} and Soundararajan, we are therefore able to establish unconditionally the correct order of magnitude of every moment less than the second of the family of quadratic Dirichlet $L$-functions. This is given in the following result.
\begin{theorem}
\label{thmupperbound}
   Unconditionally, for every $0 \leq k \leq 2$, we have
\begin{align}
\label{upperbound}
   \sumstar_{\substack{ 0<d<X \\ (d,2)=1}}|L(\tfrac{1}{2},\chi_{8d})|^k \ll_k X(\log X)^{\frac{k(k+1)}{2}}.
\end{align}
\end{theorem}

   We notice that the twisted third moment of the quadratic family of Dirichlet $L$-functions has been evaluated by M. P. Young in \cite{Young2}. Thus, if one assumes that $L(\half, \chi_{8d}) \geq 0$ for all $d$ under consideration (which follows from GRH), then we are able to obtain the following result.
\begin{theorem}
\label{thmupperbound1}
   Assume that $L(\half, \chi_{8d}) \geq 0$ for all odd, square-free $d$. Then the bound given in \eqref{upperbound} is valid for every $0 \leq k \leq 3$. In particular, this is true under GRH.
\end{theorem}

   We omit the proof of Theorem \ref{thmupperbound1} in the paper as it is similar to that of Theorem \ref{thmupperbound}. We also notice that using the approach of Radziwi{\l\l} and Soundararajan in \cite{Radziwill&Sound}, we may be able to give an alternative proof of a result of B. Hough \cite[Corollary 1.1]{Hough} that says the distribution of logarithms of central values of $L(\half, \chi_{8d})$ is bounded above by the Gaussian distribution.

  We would like to point out here that in order for the principle of Radziwi{\l\l} and Soundararajan on obtaining upper bounds for moments of $L$-functions to work, one in general needs to evaluate a given moment of $L$-functions twisted by certain character instead of just evaluating the moment itself. Previously, it is known that these twisted moments play vital roles when using mollifiers to study the non-vanishing issues of central values of $L$-functions. They are also necessary in the work of M. P. Young \cite{Young1, Young2} to use a recursive method to reduce the sizes of error terms in the moments of $L$-functions at the central point. Now, the principle of Radziwi{\l\l} and Soundararajan provides a further evidence on the importance of studying these twisted moments.

  In the last section of this paper, we propose a variant of the above mentioned principle of Radziwi{\l\l} and Soundararajan for obtaining upper bounds of moments of $L$-functions. This may potentially lead to a slightly simpler treatment (at least in some cases) when acquiring such  upper bounds.

\section{Preliminaries}
\label{sec 2}

 We include here some tools needed in our proof of Theorem \ref{thmupperbound} together with an initial treatment of the proof.
\subsection{Tools}
 From now on, we reserve the letter $p$ for a prime number and we recall the following well-known Mertens' formula (see \cite[Theorem 2.7]{MVa1}) and a consequence of it (via partial summation).
\begin{lemma} \label{RS} Let $x \geq 2$. We have, for some constant $b$,
$$
\sum_{p\le x} \frac{1}{p} = \log \log x + b+ O\Big(\frac{1}{\log x}\Big).
$$
 Also, for any integer $j \geq 1$, we have
$$
\sum_{p\le x} \frac {(\log p)^j}{p} = \frac {(\log x)^j}{j} + O((\log x)^{j-1}).
$$
\end{lemma}

   Now, we denote $\Phi$ for a smooth, non-negative function compactly supported on $[1/2,5/2]$ with $\Phi(x) =1$
for $x\in [1,2]$, and define, for any complex number $s$,
\begin{equation*}
%%\label{Phicheck}
{\widehat \Phi}(s) = \int_{0}^{\infty} \Phi(x)x^{s}\frac {dx}{x}.
\end{equation*}
   We define $\delta_{n=\square}$ to be $1$ when $n=\square$ and $0$ otherwise, where we write $\square$ for a perfect square. Similar to the proof of \cite[Proposition 1]{Radziwill&Sound}, we have the following result concerning a smoothed sum of quadratic characters.
\begin{lemma} \label{PropDirpoly}  For large $X$ and any odd positive integer $n$, we have
\begin{align*}
\sumstar_{\substack{(d,2)=1}} \chi_{8d}(n) \Phi\Big(\frac{d}{X}\Big)=
\displaystyle \delta_{n=\square}{\widehat \Phi}(1) \frac{2X}{3\zeta(2)} \prod_{p|n} \Big(\frac p{p+1}\Big) + O(X^{\frac 12+\epsilon} \sqrt{n} ).
\end{align*}
\end{lemma}

 We denote further $d(n)$ for the divisor function and $\sigma(n)$ for the sum of the positive divisors of $n$. Also define $\Lambda_j(n)$ for all integers $j \geq 0$ to be the coefficient of $n^{-s}$ in the Dirichlet series expansion of $(-1)^{j}\zeta^{(j)}(s)/\zeta(s)$. Note that this implies that $\Lambda_1(n)=\Lambda(n)$ and that $\Lambda_j(n)$ is supported on integers having at most $j$ distinct prime factors such that $\Lambda_j (n) \ll_j n^j$.

  Combining Proposition 1.1 and Proposition 1.3 in \cite{sound1} and setting $Y=X^{1/4}, M=1$ there, we readily deduce the following asymptotic result concerning the twisted second moment of quadratic Dirichlet $L$-functions.
\begin{lemma}
\label{Prop1}
  Writing any odd $l$ as  $l=l_1l^2_2$ with $l_1$ square-free, we have for any $\varepsilon>0$,
\begin{align*}
%%\label{eq:2ndmoment}
\begin{split}
 \sumstar_{(d,2)=1}L(\half, \chi_{8d})^2\chi_{8d}(l)\Phi(\frac dX)
=&  \frac{D \widehat{\Phi}(1)}{36\zeta (2)}\frac {d(l_1)}{\sqrt{l_1}}\frac {l_1}{\sigma(l_1)h(l)}X \Big (\log^3\Big (\frac {X}{l_1} \Big )-3\sum_{\substack{p | l_1}} \log^2 p\log \Big ( \frac {X}{l_1} \Big )+O(l) \Big )+O \left(X^{\frac 34+\varepsilon}l^{\half+\varepsilon}_1\right ),
\end{split}
\end{align*}
  where  $D=\frac 18\displaystyle \prod_{\substack{p \geq 3}}\left (1-\frac 1{p} \right )h(p)$ and $h$ is the multiplicative function defined on prime powers by
\begin{align*}
h(p^k)=1+\frac 1{p}+\frac 1{p^2}-\frac 4{p(p+1)}, \quad k \geq 1. 
\end{align*}
  Also,
\begin{align*}
 O(l)=& \sum^3_{j,k=0}\sum_{\substack{ m |l_1}}\sum_{\substack{n |l_1}}
\frac {\Lambda_j(m)}{m} \frac {\Lambda_k(n)}{n}D(m,n)Q_{j,k}\Big(\log \frac {X}{l_1}\Big) -3\Big (A+B\frac {\widehat{\Phi}'(1)}{\widehat{\Phi}(1)}\Big )\sum_{\substack{p | l}} \log^2 p,
\end{align*}
  where $A$ and $B$ are absolute constants and $D(m, n) \ll 1$ uniformly for all $m$ and $n$. The $Q_{j,k}$ are polynomials of degree $\leq 2$ whose coefficients involve only absolute constants and linear combinations of $\frac {\widehat{\Phi}^{(j)}(1)}{\widehat{\Phi}(1)}$ for $1 \leq j \leq 3$.
\end{lemma}

  Lastly, we define for any non-negative integer $\ell$ and any real number $x$,
\begin{equation}
\label{E_ell}
E_{\ell}(x) = \sum_{j=0}^{\ell} \frac{x^{j}}{j!}.
\end{equation}
  We recall the following two key inequalities from \cite[Lemma 1, Lemma 2]{Radziwill&Sound}.
\begin{lemma} \label{E1}  Let $\ell \geq 0$ be an even integer.  The function $E_{\ell}(x)$ is positive, convex and satisfies $E_{\ell}(x) \ge e^{x}$ for $x\le 0$.  Moreover, for $x\le \ell/e^2$, we have
$$
e^{x} \le \Big(1 + \frac{e^{-\ell}}{16}\Big) E_{\ell}(x).
$$
\end{lemma}

\begin{lemma}
\label{E2}  Let 
$x_1$, $\ldots$, $x_R$ be real numbers and let $C= \exp((e^{-\ell_1}+\ldots+e^{-\ell_R})/16)$, where $\ell_1$, $\ldots, \ell_R$ are
positive even integers. Then we have for any $y\ge 0$ and $0\le k\le 1$,  
\begin{align*}
y^{k} \le &  Cky \prod_{j=1}^{R} E_{\ell_j}((k-1)x_j) + C(1-k) \prod_{j=1}^{R} E_{\ell_j}(kx_j)
\\
&+ \sum_{r=0}^{R-1} \Big( Cky \prod_{j=1}^{r} E_{\ell_j}((k-1)x_j) + C(1-k)
\prod_{j=1}^{r} E_{\ell_j}(kx_j) \Big) \Big(\frac{e^2 x_{r+1}}{\ell_{r+1}}\Big)^{\ell_{r+1}}. 
\end{align*}
\end{lemma}

\subsection{Initial treatment of the proof of Theorem \ref{thmupperbound}}

Let $X$ be a large number and $\{ \ell_j \}_{1 \leq j \leq R}$ be a sequence of
even natural numbers defined by $\ell_1= 2\lceil 100 \log \log X\rceil$ and $\ell_{j+1} = 2 \lceil 100 \log \ell_j \rceil$ for $j \geq 1$.  Here we choose $R$ to be the largest natural number such that $\ell_R >10^4$ and observe that we have $\ell_{j} > \ell_{j+1}^2$ for all $1 \leq j \leq R-1$. Let ${ P}_1$ be the set of odd primes below $X^{1/\ell_1^2}$.  For $2\le j\le R$, we define
${ P_j}$ to be the set of primes lying in the interval $(X^{1/\ell_{j-1}^2}, X^{1/\ell_j^2}]$.   We also define
\begin{equation*}
%% \label{defP}
{\mathcal P}_j(d) = \sum_{p\in P_j} \frac{1}{\sqrt{p}} \chi_{8d}(p).
\end{equation*}
 Given a real number $\alpha$, we denote
\begin{align}
\label{defN}
{\mathcal N}_j(d, \alpha) = E_{\ell_j} (\alpha {\mathcal P}_j(d)), \quad \mathcal{N}(d, \alpha) = \prod_{j=1}^{R} {\mathcal N}_j(d,\alpha).
\end{align}
 We further set for two real numbers $n, k$ satisfying $0\le k \le 1$,
\begin{equation*}
%%\label{defA}
{\mathcal A}_j(d) = {\mathcal N}_j(d, (k-1) n ), \quad {\mathcal B}_j(d) = {\mathcal N}_j(d, nk ).
\end{equation*}

We apply Lemma \ref{E2} with $y=|L(\frac 12, \chi_{8d})|^n(\log d)^{-\frac n2}$ and $x_j =n{\mathcal P}_j(d)$ to obtain the following bounds for the moments of $L(\frac 12,\chi_{8d})$.
\begin{proposition} \label{Prop2}  With notations as above, we have
\begin{align}
\begin{split}
\label{upperboundforLmoment}
\Big(|L(\frac 12,\chi_{8d})|^n(\log d)^{-\frac n2} \Big)^k \le &   Ck |L(\frac12, \chi_{8d})|^n(\log d)^{-\frac n2}
\Big( \prod_{j=1}^{R} {\mathcal A}_j(d) + \sum_{r=0}^{R-1} \prod_{j=1}^{r} {\mathcal A}_j(d) \Big( \frac{e^2 n {\mathcal P}_{r+1} (d)}{\ell_{r+1}} \Big)^{\ell_{r+1}} \Big)
\\
&+ C(1-k) \Big( \prod_{j=1}^{R} {\mathcal B}_j(d) + \sum_{r=0}^{R-1}\prod_{j=1}^{r} {\mathcal B}_j(d) \Big(\frac{e^2 n {\mathcal P}_{r+1}(d)}{\ell_{r+1}}\Big)^{\ell_{r+1}}\Big).
\end{split}
\end{align}
\end{proposition}

  Arguing similar to \cite[Section 3.4]{Radziwill&Sound}, we see that in order to prove Theorem \ref{thmupperbound}, it suffices to show that the right side of \eqref{upperboundforLmoment} averaged over $d$ is $\ll X (\log X)^{\frac{(n k)^2}{2}}$. We close this section by giving such an estimation for the terms involving with ${\mathcal B}_j(d)$.

\begin{proposition} \label{Prop3} With notations as above, we have
$$
\sumstar_{(d,2)=1} \Big( \prod_{j=1}^{R} {\mathcal B}_j(d) +
\sum_{r=0}^{R-1} \prod_{j=1}^r {\mathcal B}_j(d) \Big(\frac{e^2 n {\mathcal P}_{r+1}(d)}{\ell_{r+1}}\Big)^{\ell_{r+1}} \Big) \Phi\Big(\frac{ d}{X}\Big)
\ll X (\log X)^{\frac{(n k)^2}{2}}.
$$
\end{proposition}
\begin{proof}
 Let $w(n)$
be the multiplicative function defined by $w(p^{\alpha}) = \alpha!$ for prime powers $p^{\alpha}$ and let $\Omega(n)$ denote
the number of distinct prime powers dividing $n$.  We also define functions $b_j(n), p_j(n)$ for $1 \leq j \leq R$ such that $b_j(n), p_j(n)=0$ or $1$, and we have $b_j(n)=1$ ($p_j(n)=1$) if and only if $n$ is composed of at most (exactly, counted with multiplicity) $\ell_j$ primes, all from the interval $P_j$ . Using these notations, we see that
\begin{equation}
\label{5.1}
{\mathcal B}_j(d) = \sum_{n_j} \frac{1}{\sqrt{n_j}} \frac{(nk)^{\Omega(n_j)}}{w(n_j)}  b_j(n_j) \chi_{8d}(n_j), \quad \frac{1}{\ell_j!} {\mathcal P}_j(d)^{\ell_j} = \sum_{n_j} \frac{1}{w(n_j)\sqrt{n_j}} p_j(n_j) \chi_{8d}(n_j), \quad 1\le j\le R.
\end{equation}
    We note here that both ${\mathcal B}_j(d)$ and ${\mathcal P}_j(d)^{\ell_j}$ are short Dirichlet polynomials since $b_j(n_j), p_j(n_j)=0$ unless $n_j \leq (X^{1/\ell_j^2})^{\ell_j}=X^{1/\ell_j}$. It follows that the expressions $\prod_{j=1}^{R} {\mathcal B}_j(d), \prod_{j=1}^{r} {\mathcal B}_j(d) {\mathcal P}_{r+1}^{\ell_{r+1}}(d)$ are all short Dirichlet polynomials of length at most $X^{1/\ell_1+ \ldots +1/\ell_R} < X^{1/1000}$.

 We expand the term $\prod_{j=1}^{r} {\mathcal B}_j(d) {\mathcal P}_{r+1}^{\ell_{r+1}}(d)$ for some $0\le r\le R-1$ using \eqref{5.1} and apply Lemma  \ref{PropDirpoly} to estimate it. By doing so, we may ignore the error term in Lemma \ref{PropDirpoly} as both ${\mathcal B}_j$ and ${\mathcal P_j}^{\ell_j}(d)$ are short Dirichlet polynomials. Considering the main term contributions from Lemma \ref{PropDirpoly}, we see that
\begin{align*}
 \prod_{j=1}^{r} {\mathcal B}_j(d) {\mathcal P}_{r+1}^{\ell_{r+1}} \ll X
 &\prod_{j=1}^{r} \Big( \sum_{n_j =\square} \frac{1}{\sqrt{n_j}} \frac{(nk)^{\Omega(n_j)}}{w(n_j)} \prod_{p|n_j} \Big( \frac p{p+1}\Big) b_j(n_j)\Big)\\
 &\times  \Big( \ell_{r+1}! \sum_{n_{r+1}=\square} \frac{1}{w(n_{r+1}) \sqrt{n_{r+1}}}
\prod_{p|n_{r+1}} \Big(\frac p{p+1} \Big) p_{r+1}(n_{r+1}) \Big).
\end{align*}

 The proof of the proposition now follows by arguing in the same way as in the proof of Proposition 4 in \cite{Radziwill&Sound}.
\end{proof}

\section{Proof of Theorem \ref{thmupperbound}}

  In view of our discussions in the previous section, it remains to show that the right side of \eqref{upperboundforLmoment} averaged over $d$  for the terms involving with ${\mathcal A}_j(d)$ is also $\ll X (\log X)^{\frac{(n k)^2}{2}}$ for $n=2$. As our approach here may be applied to treat other values of $n$, we shall retain the symbol $n$ in most of the places in the rest of the section instead of specitying it to be $2$. Thus, to conclude the proof of Theorem \ref{thmupperbound}, it suffices to establish the following result.
\begin{proposition} \label{Prop4} With notations as above, we have for $n=2$,
\begin{align}
\label{Aestmation}
\sumstar_{(d,2)=1} |L(\frac 12,\chi_{8d})|^n \Big(  \prod_{j=1}^{R} {\mathcal A}_j(d) +
\sum_{r=0}^{R-1} \prod_{j=1}^r {\mathcal A}_j(d) \Big(\frac{e^2n {\mathcal P}_{r+1}(d)}{\ell_{r+1}}\Big)^{\ell_{r+1}} \Big) \Phi\Big(\frac{ d}{X}\Big)
\ll X (\log X)^{ \frac {(nk)^2+n}{2}}.
\end{align}
\end{proposition}
 In the remaining of this section, we give a proof of Proposition \ref{Prop4}. First note that we may replace $|L(\frac 12,\chi_{8d})|^n$ by $L(\frac 12,\chi_{8d})^n$ when $n=2$ since $L(\frac 12,\chi_{8d})$ is real.  As the arugments are similar, it suffices to show that
\begin{align}
\label{estAP}
\sumstar_{(d,2)=1} &L(\tfrac 12,\chi_{8d})^n \sum_{r=0}^{R-1}  \prod_{j=1}^{r} {\mathcal A}_j(d)
\Big(\frac{e^2n {\mathcal P}_{r+1}(d)}{\ell_{r+1}}\Big)^{\ell_{r+1}} \Phi\Big(\frac{d}{X}\Big) \ll X (\log X)^{ \frac {(nk)^2+n}{2}}.
\end{align}

Note first that we have
$$
{\mathcal A}_j(d) = \sum_{n_j} \frac{1}{\sqrt{n_j}} \frac{(n (k-1))^{\Omega(n_j)}}{w(n_j)} b_j(n_j) \chi_{8d}(n_j), \quad 1\le j\le R.
$$
 Analogue to our discussions above, the product
$\prod_{j=1}^r {\mathcal A}_j(d) {\mathcal P}_{{r+1}}^{\ell_{r+1}}(d)$ for all $0\le r\le R-1$ are short Dirichlet polynomials of
length at most $X^{1/1000}$.

  We now apply Lemma \ref{Prop1} to evaluate $\prod_{j=1}^{r} {\mathcal A}_j(d) {\mathcal P}_{r+1}^{\ell_{r+1}}$(d) for some $0\le r\le R-1$ by
expanding it into Dirichlet series. Once again we may focus only on the main term to see that, upon writing $n_j=(n_j)_1(n_j)_{2}^2$ with $(n_j)_{1}$ being square-free, 
\begin{align}
\label{Aestmationexpansion}
\begin{split}
& \sumstar_{(d,2)=1} L(\tfrac 12,\chi_{8d})^n \prod_{j=1}^{r} {\mathcal A}_j(d)
{\mathcal P}_{r+1}^{\ell_{r+1}}(d) \Phi\Big(\frac{d}{X}\Big) \\
\ll & X   \sum_{n_1, \cdots, n_{r+1} }  \Big( \prod_{j=1}^{r} \frac{1}{\sqrt{n_j(n_j)_{1}}}
\frac{(n (k-1))^{\Omega(n_j)}}{w(n_j) } b_j(n_j) \frac {d((n_j)_{1})(n_j)_{1}}{\sigma((n_j)_{1})h(n_{j})}  \Big) \Big(   \frac{\ell_{r+1}!}{\sqrt{n_{r+1}(n_{r+1})_1}} \frac{p_{r+1}(n_{r+1})}{w(n_{r+1})}\frac {d((n_{r+1})_{1})(n_{r+1})_{1}}{\sigma((n_{r+1})_{1})h(n_{r+1})} \Big) \\
& \times \Big (\log^3\Big ( \frac {X}{(n_1)_1 \cdots (n_{r+1})_1} \Big )-3\sum_{\substack{p | (n_1)_1 \cdots (n_{r+1})_1 }} \log^2 p\log \Big (\frac {X}{(n_1)_1 \cdots (n_{r+1})_1} \Big )+O(n_1 \cdots n_{r+1}) \Big ).
\end{split}
\end{align}
  As the estimations are similar, we may consider the above sums involving with the terms $\log^3(X/((n_1)_1 \cdots (n_{r+1})_1))=(\log X-\log ((n_1)_1 \cdots (n_{r+1})_1))^3$  only. Upon expanding, we observe that we can write $\log^3(X/((n_1)_1 \cdots (n_{r+1})_1))$ as linear combinations of the sum:
\begin{align}
\label{typicalform}
  C(m_0, \ldots, m_{r+1})(\log X)^{m_0} \sum_{\substack{ p_i | (n_i)_1 \\ 1 \leq i \leq r+1}}  \prod_{1 \leq i \leq r+1} (\log p_i)^{m_i},
\end{align}
   where $m_j, 0 \leq j \leq r+1$ are non-negative integers satisfying $\sum_{0 \leq j \leq r+1}m_j=3$ and where $C(m_0, \ldots, m_{r+1})$ are bounded constants. Without loss of generality, we may group terms to consider the total contribution to \eqref{Aestmationexpansion} from all terms of the above form corresponding to $m_{i_1}=m_{i_2}=m_{i_3}=1$ for some $1 \leq i_1 < i_2< i_3 \leq r+1$. For example, when the corresponding $p_{i_1} \in P_1, p_{i_2} \in P_2, p_{i_3} \in P_{r+1}$, the contribution is
\begin{align}
\label{6.1}
\begin{split}
 \ll & \sum_{l_1, l_2, l_3 \geq 0}  \prod^3_{s=1}\Big( \frac{\log p_{i_s}}{p^{l_s+1}_{i_s}}\frac{|(n (k-1))|^{2l_s+1}}{(2l_s+1)!} \frac {n p_{i_s}}{(p_{i_s}+1)h(p^{2l_s+1}_{i_s})}\Big) \\
& \times \prod_{j=1}^{r} \Big( \sum_{(n_j, p_{i_1}p_{i_2})=1  } \frac{1}{\sqrt{n_j (n_{j})_{1}}}
\frac{(n (k-1))^{\Omega(n_j)}}{w(n_j) } \widetilde{b}_{j, l_1, l_2}(n_j) \frac {d((n_j)_{1})(n_j)_{1}}{\sigma((n_j)_{1})h(n_{j})} \Big)  \\
&\times \Big( \ell_{r+1}! \sum_{(n_{r+1}, p_{i_3})=1 } \frac{1}{\sqrt{n_{r+1}(n_{r+1})_{1}}} \frac{p_{r+1}(n_{r+1}p^{2l_3+1}_{i_3})}{w(n_{r+1})}\frac {d((n_{r+1})_{1})(n_{r+1})_{1}}{\sigma((n_{r+1})_{1})h(n_{r+1})} \Big),
\end{split}
\end{align}
  where we define $\widetilde{b}_{j, l_1, l_2}(n_j)=b_j(n_jp^{l_j}_{i_j})$ for $j=1,2$ and $\widetilde{b}_{j, l_1, l_2}(n_j)=b_{j}(n_j)$ otherwise.

  Let us consider the sum over $n_1$ in \eqref{6.1}.  If we replace the factor $\widetilde{b}_{l, l_1, l_2}(n_1)$ by $1$, then the sum becomes
\begin{align}
\label{6.2}
\begin{split}
& \prod_{\substack{p\in P_1 \\ (p,p_{i_1})=1}} \Big( \sum_{j=0}^{\infty} \frac{1}{p^j} \frac{(n (k-1))^{2j}}{(2j)!h(p^{2j})} + \sum_{j=0}^{\infty} \frac{1}{p^{j+1}}
\frac{(n (k-1))^{2j+1}}{(2j+1)!}\frac {n p}{(p+1)h(p^{2j+1})}\Big) \\
\ll & \Big ( \prod_{\substack{p\in P_1 \\ (p,p_{i_1})=1}}C(p) \Big ) \times \exp \Big ( \big(\frac {(n (k-1))^{2}}{2}+ n^2 (k-1) \big )\sum_{\substack{p\in P_1}}\frac 1p \Big ),
\end{split}
\end{align}
  where for some constant $A$ independent of $p$, 
\begin{align*}
%%\label{Cp}
\begin{split}
 C(p)=& \exp(\frac A{p^2}) \Big (1- \big(\frac {(n (k-1))^{2}}{2}+n^2 (k-1) \big )\frac 1p \Big ) \Big( \sum_{j=0}^{\infty} \frac{1}{p^j} \frac{(n (k-1))^{2j}}{(2j)!h(p^{2j})}  + \sum_{j=0}^{\infty} \frac{1}{p^{j+1}}
\frac{(n (k-1))^{2j+1}}{(2j+1)!}\frac {n p}{(p+1)h(p^{2j+1})}\Big).
\end{split}
\end{align*}
   Note here that $C(p)$ is well-defined as one checks readily checked that each factor in the above product is positive for $n= 2, 0 \leq k \leq 1$ and $p \geq 3$. Meanwhile, we note that the left side of \eqref{6.2} is also 
\begin{align}
\label{Eulerprodlowerbound}
\begin{split}
 \gg \exp \Big ( \big(\frac {(n (k-1))^{2}}{2}+ n^2 (k-1) \big )\sum_{\substack{p\in P_1}}\frac 1p \Big ).
\end{split}
\end{align}

  On the other hand, using Rankin's trick by noticing that $2^{\Omega(n_1)-\ell_1}\ge 1$ if $\Omega(n_1) > \ell_1$,  we see that the error introduced by replacing $\widetilde{b}_{l, l_1, l_2}(n_1)$ with $1$ does not exceed
\begin{align}
\label{errorbound}
\begin{split}
 & \sum_{n_1} \frac{1}{\sqrt{n_1(n_{1})_{1}}} 
\frac{|n(k-1)|^{\Omega(n_1)}}{w(n_1) } 2^{\Omega(n_1)-\ell_1} \frac {d((n_1)_{1})(n_1)_{1}}{\sigma((n_1)_{1})h(n_{1})} \\
\le & 2^{-\ell_1} \prod_{\substack{p\in P_1 \\ (p,p_{i_1})=1}} \Big( 1+\sum_{j=1}^{\infty} \frac{1}{p^j} \frac{(n (k-1))^{2j}2^{2j}}{(2j)!} \frac {1}{h(p^{2j})} + \sum_{j=0}^{\infty} \frac{1}{p^{j+1}}
\frac{|n (k-1)|^{2j+1}2^{2j+1}}{(2j+1)!}\frac {n p}{(p+1)h(p^{2j+1})}\Big) \\
\ll & 2^{-\ell_1}
\exp\Big( \big (2(n(k-1))^2+2n^2(1-k)\big ) \sum_{p\in P_1} \frac{1}{p} \Big).
\end{split}
\end{align}

  We deduce from \eqref{6.2}, \eqref{Eulerprodlowerbound} and \eqref{errorbound} that the error term is
\begin{align}
\label{errorbound1}
\ll & 2^{-\ell_1}
\exp\Big( \big (\frac 32(n(k-1))^2+3n^2(1-k)\big ) \sum_{p\in P_1} \frac{1}{p} \Big) \Big ( \prod_{\substack{p\in P_1 \\ (p,p_{i_1})=1}}C(p) \Big ) \times \exp \Big ( \big(\frac {(n (k-1))^{2}}{2}+n^2 (k-1) \big )\sum_{\substack{p\in P_1}}\frac 1p \Big ).
\end{align}

 Note that Lemma \ref{RS} implies that for all $1 \leq j \leq R$, we have $\sum_{p\in P_j} 1/p \leq 2\log \ell_{j-1} \leq \ell_j/36$ from our definition on $\ell_j$. We obtain from this and \eqref{errorbound1} that for $n = 2$, we have
\begin{align}
\label{6.3}
\begin{split}
 & \sum_{(n_1, p_{i_1})=1  } \frac{1}{\sqrt{n_1 (n_{1})_1}}
\frac{(n (k-1))^{\Omega(n_1)}}{w(n_1) } \widetilde{b}_{1, l_1, l_2}(n_1) \frac {d((n_1)_{1})(n_1)_{1}}{\sigma((n_1)_{1})h(n_{1})} \\
\ll &   (1+O(2^{-\ell_1/2}))\Big ( \prod_{\substack{p\in P_1 \\ (p,p_{i_1})=1}}C(p) \Big ) \times \exp \Big ( \big(\frac {(n (k-1))^{2}}{2}+n^2 (k-1) \big )\sum_{\substack{p\in P_1}}\frac 1p \Big ) .
\end{split}
\end{align}
 We may also establish similar estimations for sums over $n_j, 2 \leq j \leq r$.

Next, we apply Rankin's trick again to see that the contribution of the $n_{r+1}$ terms in \eqref{6.1} is
$$
\le \ell_{r+1}! 10^{-\ell_{r+1}} \prod_{p\in P_{r+1}}
\Big( \sum_{j=0}^{\infty}\frac{10^{2j}}{p^j (2j)!h(p^{2j})}+ \sum_{j=1}^{\infty}
\frac{ 10^{2j+1} }{p^{j+1} (2j+1)!} \frac {n p}{(p+1)h(p^{2j+1})}\Big).
$$
We apply Lemma \ref{RS}, the estimation $\ell_{r+1}! \le \ell_{r+1} (\ell_{r+1}/e)^{\ell_{r+1}}$ and the definition of $\ell_{r+1}$ to see that the above is
$$
\ll \ell_{r+1} \Big(\frac{\ell_{r+1}}{10e}\Big)^{\ell_{r+1}} \exp\Big(70 \sum_{p\in P_{r+1}} \frac{1}{p}\Big)
 \ll \ell_{r+1} \Big(\frac{\ell_{r+1}}{10e}\Big)^{\ell_{r+1}} \exp(\tfrac 57 \ell_{r+1}).
$$

Combining the this with \eqref{6.1} and \eqref{6.3}, we conclude that the contribution from all terms of the forms given in \eqref{typicalform} corresponding to $m_{i_1}=m_{i_2}=m_{i_3}=1$ for some $1 \leq i_1 < i_2< i_3 \leq r+1$ to the left side of \eqref{estAP} is
\begin{align*}
\ll & {X} e^{-\ell_{r+1}/7}\prod_{1 \leq j \leq r} \Big(1+O(2^{-\ell_j/2})\Big ) \prod_{\substack{p\in \bigcup_{j=1}^{r} P_j}}C(p)  \times \exp \Big ( \big(\frac {(n (k-1))^{2}}{2}+n^2 (k-1) \big )\sum_{\substack{p\in \bigcup_{j=1}^{r} P_j }}\frac 1p \Big ) \\
& \times \Big(\sum_{p\in \bigcup_{j=1}^{r} P_j}\sum_{l\geq 0}  \frac{ \log p}{p^{l+1}}\frac{|(n (k-1))|^{2l+1}}{(2l+1)!} \frac {n p}{(p+1)h(p^{2l+1})} \Big )^3  \\
\ll & e^{-\ell_{r+1}/7} X (\log X)^{ \frac{(n (k-1))^{2}}{2}+n^2 (k-1)+3 },
\end{align*}
 by using Lemma \ref{RS} while noting that $\displaystyle \prod_p C(p) \ll 1$ and $\displaystyle \prod_{1 \leq j \leq r} \Big(1+O(2^{-\ell_j/2})\Big ) \ll 1$ since $\ell_{j} > \ell_{j+1}^2$.

  Summing over $r$, we deduce that the estimation in \eqref{estAP} is valid and this completes the proof of the proposition.

\section{A further remark}

  In this section we describe a variant on the principle of Radziwi{\l\l} and Soundararajan mentioned in Section \ref{sec 1} for obtaining upper bounds of moments of $L$-functions. This is based on the following simple observation.
\begin{lemma}
\label{E0}  Let $\ell$ be a non-negative even integer.  For all real number $x$, let $E_{\ell}(x)$ be defined in \eqref{E_ell}. We have
$$
 E_{\ell}(x)E_{\ell}(-x) \geq 1.
$$
\end{lemma}
\begin{proof}
  We expand out the product $f_{\ell}(x) :=E_{\ell}(x)E_{\ell}(-x)-1$ as a polynomial to see that it suffices to show that the coefficients of $x^i, 0 \leq j \leq 2\ell$ are all non-negative. Note that the function $f_{\ell}(x)$ is even so that only even powers of $x$ appear in the expansion. Furthermore, it follows from the binomial theorem that the only non-zero even powers of $x$ involved are those of the form $x^{2j}$ for $2j > l$. We then set $\ell=2k$ to see that the coefficient of $x^{2(k+j)}$ for some integer $j>0$ is given by
$$
  \frac 1{(2(k+j))!}\sum^{2k}_{i=2j} \binom {2(k+j)}{i}(-1)^i=  -\frac 2{(2(k+j))!}\sum^{2j-1}_{i=0} \binom {2(k+j)}{i}(-1)^i \geq 0,
$$
  where the last inequality above follows from \cite[(6.6)]{FI10} and this completes the proof.
\end{proof}

  Now, for two real numbers $n, k$ with $0<k<1$, we apply Lemma \ref{E0} to see that
\begin{align*}
%%\label{defPd}
 & \sumstar_{(d,2)=1}|L(\half, \chi_{8d})|^{nk}  \Phi(\frac dX) \leq  \sumstar_{(d,2)=1}|L(\half, \chi_{8d})|^{nk}
 {\mathcal N}(d, n(k-1))^{k}{\mathcal N}(d, n(1-k))^{k} \Phi(\frac dX),
\end{align*}
 where we recall that the definition of ${\mathcal N}(d, \alpha)$ is given in \eqref{defN}.
  We further apply H\"older's inequality to the right side expression above to deduce that
\begin{align*}
%%\label{defPd}
 & \sumstar_{(d,2)=1}|L(\half, \chi_{8d})|^{nk}  \Phi(\frac dX)  \leq  \Big ( \sumstar_{(d,2)=1}|L(\half, \chi_{8d})|^n {\mathcal N}(d, n(k-1))  \Phi(\frac dX)\Big )^{k}\Big ( \sumstar_{(d,2)=1}{\mathcal N}(d, n(1-k))^{k/(1-k)} \Phi(\frac dX) \Big)^{1-k}.
\end{align*}

   This approach can be applied to obtain upper bounds for the $k$-th moment of $L$-functions. In particalur, it is convenient to study the $\half$-th moment as it suffices to evaluate the average over $d$ of $|L(\half, \chi_{8d})| {\mathcal N}(d, -\half)$ and ${\mathcal N}(d, \half)$ by taking $n=1, k=1/2$ above. If we assume that $L(\half, \chi_{8d}) \geq 0$ (which follows from GRH), then we just need to consider $L(\half, \chi_{8d}) {\mathcal N}(d, -\half)$,  which is relatively simpler compared to our work above. The same method applies equally to the study on the moments of quadratic twists of $L$-functions attached to elliptic curves (as we know in this case the corresponding $L$-functions have non-negative values at the central point).

   More generally, when $k=1/m$ with $m>2$ a positive integer, we denote $\{ L(s, f)  \}_{f \in \mathcal F}$ for a general family of $L$-functions, we may apply Lemma \ref{E0} (with a suitable adjustment on the defintion of ${\mathcal N}$) to see that
\begin{align*}
%%\label{defPd}
 & \sum_{f \in \mathcal F}|L(\half, f)|^{k}
\leq    \sum_{f \in \mathcal F}|L(\half, f)|^{k}
 {\mathcal N}(f, k-1)^{k}\Big ( \prod^{m-2}_{i=1}\big ({\mathcal N}(f, 1-ik){\mathcal N}(f, (i+1)k-1)\big)^{k} \Big ){\mathcal N}(f, 1-(m-1)k)^{k}  .
\end{align*}
  Applying H\"older's inequality $m-1$ times to the right side expression above, we deduce that
\begin{align*}
%%\label{defPd}
\sum_{f \in \mathcal F}|L(\half, f)|^{k} \leq &  \Big ( \sum_{f \in \mathcal F}|L(\half, f)| {\mathcal N}(f, k-1) \Big )^{k}\prod^{m-2}_{i=1}\Big ( \sum_{f \in \mathcal F}{\mathcal N}(f, 1-ik){\mathcal N}(f, (i+1)k-1)  \Big)^{k}  \Big ( \sum_{f \in \mathcal F}{\mathcal N}(f, 1-(m-1)k)  \Big)^{k}.
\end{align*}
 It follows that in order to obtain upper bounds for the $1/m$-th moment of the corresponding family, we only need to be able to evaluate the average over $f \in \mathcal F$ of $|L(\half, f)| {\mathcal N}(f, k-1)$ and quantities involving with products of at most two copies of ${\mathcal N}$. The above approach can also be adapted to treat moments of the Riemann zeta function on the critical line. We shall however not go any further in this direction here.

\vspace*{.5cm}

\noindent{\bf Acknowledgments.} P. G. is supported in part by NSFC grant 11871082.

\bibliography{biblio}

\begin{bibdiv}
\begin{biblist}

\bib{Andrade-Keating01}{article}{
      author={Andrade, J.~C.},
      author={Keating, J.~P.},
       title={Conjectures for the integral moments and ratios of
  {$L$}-functions over function fields},
        date={2014},
     journal={J. Number Theory},
      volume={142},
       pages={102\ndash 148},
}

\bib{CFKRS}{article}{
      author={Conrey, J.~B.},
      author={Farmer, D.~W.},
      author={Keating, J.~P.},
      author={Rubinstein, M.~O.},
      author={Snaith, N.~C.},
       title={Integral moments of {$L$}-functions},
        date={2005},
     journal={Proc. London Math. Soc. (3)},
      volume={91},
      number={1},
       pages={33\ndash 104},
}

\bib{FI10}{book}{
      author={Friedlander, J.},
      author={Iwaniec, H.},
       title={Opera de cribro},
      series={American Mathematical Society Colloquium Publications},
   publisher={American Mathematical Society, Providence, RI},
        date={2010},
      volume={57},
}

\bib{DoHo}{article}{
      author={Goldfeld, D.},
      author={Hoffstein, J.},
       title={Eisenstein series of $\frac{1}{2}$-integral weight and the mean
  value of real {D}irichlet {$L$}-series},
        date={1985},
     journal={Invent. Math.},
      volume={80},
      number={2},
       pages={185\ndash 208},
}

\bib{Harper}{article}{
      author={Harper, A.~J.},
       title={Sharp conditional bounds for moments of the {R}iemann zeta
  function},
        date={Preprint},
        note={arXiv:1305.4618},
}

\bib{H&Sound}{article}{
      author={Heap, W.},
      author={Soundararajan, K.},
       title={Lower bounds for moments of zeta and {$L$}-functions revisited},
        date={Preprint},
        note={arXiv:2007.13154},
}

\bib{Hough}{article}{
      author={Hough, B.},
       title={The distribution of the logarithm in an orthogonal and a
  symplectic family of {$L$}-functions},
        date={2014},
        ISSN={0933-7741},
     journal={Forum Math.},
      volume={26},
      number={2},
       pages={523\ndash 546},
}

\bib{Jutila}{article}{
      author={Jutila, M.},
       title={On the mean value of ${L}(1/2,\chi)$ for real characters},
        date={1981},
     journal={Analysis},
      volume={1},
      number={2},
       pages={149\ndash 161},
}

\bib{Keating-Snaith02}{article}{
      author={Keating, J.~P.},
      author={Snaith, N.~C.},
       title={Random matrix theory and {$L$}-functions at {$s=1/2$}},
        date={2000},
     journal={Comm. Math. Phys.},
      volume={214},
      number={1},
       pages={91\ndash 110},
}

\bib{MVa1}{book}{
      author={Montgomery, H.~L.},
      author={Vaughan, R.~C.},
       title={{M}ultiplicative number theory. {I}. {C}lassical theory},
      series={Cambridge Studies in Advanced Mathematics},
   publisher={Cambridge University Press},
     address={Cambridge},
        date={2007},
      volume={97},
}

\bib{Radziwill&Sound1}{article}{
      author={Radziwi{\l\l}, M.},
      author={Soundararajan, K.},
       title={Continuous lower bounds for moments of zeta and {$L$}-functions},
        date={2013},
     journal={Mathematika},
      volume={59},
      number={1},
       pages={119\ndash 128},
}

\bib{Radziwill&Sound}{article}{
      author={Radziwi{\l \l}, M.},
      author={Soundararajan, K.},
       title={Moments and distribution of central {$L$}-values of quadratic
  twists of elliptic curves},
        date={2015},
     journal={Invent. Math.},
      volume={202},
      number={3},
       pages={1029\ndash 1068},
}

\bib{R&Sound}{article}{
      author={Rudnick, Z.},
      author={Soundararajan, K.},
       title={Lower bounds for moments of {$L$}-functions},
        date={2005},
     journal={Proc. Natl. Acad. Sci. USA},
      volume={102},
      number={19},
       pages={6837\ndash 6838},
}

\bib{R&Sound1}{incollection}{
      author={Rudnick, Z.},
      author={Soundararajan, K.},
       title={Lower bounds for moments of {$L$}-functions: symplectic and
  orthogonal examples},
        date={2006},
   booktitle={Multiple {D}irichlet series, automorphic forms, and analytic
  number theory},
      series={Proc. Sympos. Pure Math.},
      volume={75},
   publisher={Amer. Math. Soc., Providence, RI},
       pages={293\ndash 303},
}

\bib{Shen}{article}{
      author={Shen, Q.},
       title={The fourth moment of quadratic {D}irichlet {$L$}-functions},
        date={to appear},
     journal={Math. Z.},
}

\bib{Sono}{article}{
      author={Sono, K.},
       title={The second moment of quadratic {D}irichlet {$L$}-functions},
        date={2020},
     journal={J. Number Theory},
      volume={206},
       pages={194\ndash 230},
}

\bib{sound1}{article}{
      author={Soundararajan, K.},
       title={Nonvanishing of quadratic {D}irichlet {$L$}-functions at
  $s=\frac{1}{2}$},
        date={2000},
     journal={Ann. of Math. (2)},
      volume={152},
      number={2},
       pages={447\ndash 488},
}

\bib{Sound01}{article}{
      author={Soundararajan, K.},
       title={Moments of the {R}iemann zeta function},
        date={2009},
     journal={Ann. of Math. (2)},
      volume={170},
      number={2},
       pages={981\ndash 993},
}

\bib{ViTa}{article}{
      author={Vinogradov, A.~I.},
      author={Takhtadzhyan, L.~A.},
       title={Analogues of the {V}inogradov-{G}auss formula on the critical
  line},
        date={1981},
     journal={Zap. Nauchn. Sem. Leningrad. Otdel. Mat. Inst. Steklov. (LOMI)},
      volume={109},
       pages={41\ndash 82, 180\ndash 181, 182\ndash 183},
}

\bib{Young1}{article}{
      author={Young, M.~P.},
       title={The first moment of quadratic {D}irichlet {$L$}-functions},
        date={2009},
     journal={Acta Arith.},
      volume={138},
      number={1},
       pages={73\ndash 99},
}

\bib{Young2}{article}{
      author={Young, M.~P.},
       title={The third moment of quadratic {D}irichlet {L}-functions},
        date={2013},
     journal={Selecta Math. (N.S.)},
      volume={19},
      number={2},
       pages={509\ndash 543},
}

\end{biblist}
\end{bibdiv}
\bibliographystyle{amsxport}

\vspace*{.5cm}

\end{document}